\documentclass[a4paper,reqno]{amsart}
\usepackage{amsmath}
\usepackage{amsfonts}
\usepackage{amssymb}
\usepackage{xcolor}

\newtheorem{theorem}{Theorem}[section]
\newtheorem{lemma}[theorem]{Lemma}
\newtheorem{definition}[theorem]{Definition}
\newtheorem{proposition}[theorem]{Proposition}
\newtheorem{corollary}[theorem]{Corollary}
\newtheorem{remark}[theorem]{Remark}

\newtheorem{conjecture}[theorem]{Conjecture}

\newcommand{\RR}{\mathbb{R}}
\newcommand{\FF}{\mathbb{F}}

\newcommand{\PP}{\mathbb{P}}

\newcommand{\F}{\mathbb{F}_{p}^{n}}
\newcommand{\E}{\widehat{E}}
\newcommand{\m}{|}
\newcommand{\e}{\mathcal{E}}
\newcommand{\ee}{e^{-\frac{2\pi i  x \cdot\xi}{p}}}

\begin{document}

\title{Projections in vector spaces over finite fields}

\author{Changhao Chen}

\address{School of Mathematics and Statistics, The University of New South Wales, Sydney NSW 2052, Australia }
\email{changhao.chenm@gmail.com}

\date{\today}

\subjclass[2010]{05B25, 52C99}
\keywords{Exceptional sets of projections, finite fields, random subsets, Fourier transform}


\begin{abstract}
We study the projections in vector spaces over finite fields. We prove finite fields  analogues of the bounds on the dimensions of the exceptional sets  for Euclidean projection mapping. We provide examples which do not have  exceptional projections via  projections of random sets. In the end, we study  the projections of sets  which have the  (discrete) Fourier decay.    
\end{abstract}


\maketitle

\section{Introduction}

A fundamental problem in fractal geometry is that how the projections affect dimension. Recall the classical Marstand-Mattila projection theorem: Let $E\subset \RR^{n}, n\geq2,$ be a Borel set with Hausdorff dimension $s$. 
\begin{itemize}
\item If $s\leq m$, then the orthogonal projection of $E$ onto almost all $m$-dimensional subspaces has Hausdorff dimension $s$.

\item If $s>m$, then the orthogonal of $E$ onto almost all $m$-dimensional subspaces has positive $m$-dimensional Lebesgue measure. 
\end{itemize}
In 1954 Marstand \cite{Marstrand} proved this projection theorem
in the plane. In 1975 Mattila \cite{Mattila1975} 
proved this  for general dimension via   1968  Kaufman's  \cite{Kaufman} potential theoretic methods. Furthermore  the bounds on the dimensions of the exceptional sets  of projections was well studied. We put these results in the following form Theorem \ref{thm:M} which come from  a recent survey paper of Falconer, Fraser, and Jin \cite{Falconer}, and Mattila \cite[Corollary 5.12]{Mattila}. 
Theorem \ref{thm:M} $(a)$ was proved by Kaufman \cite{Kaufman} for $m=1, n=2$, and by Mattila \cite{Mattila1975} for general $1\leq m<n$. The estimate  $(b)$ was proved by Falconer \cite{Falconer1982} and all known proofs depend on Fourier transform. The estimate $(c)$ was proved by Peres and Schlag \cite{Peres} under their generalized projections.

\begin{theorem}[Bounds on the dimensions of the exceptional sets]\label{thm:M}
Let $E \subset \RR^{n}$ be a Borel set and $s=\dim E$.

(a) If $s\leq m$ and $t \in (0, s]$, then
\[
\dim \{V \in  G(n,m) : \dim \pi_{V} (E) < t \} \leq m(n - m) -(m-t).
\]

(b) If $s> m$, then
\[
\dim \{V \in G(n,m) : \mathcal{H}^{m}(\pi_{V} (E)) = 0\} \leq  m(n - m) -(s-m).
\]

(c) If $s>2m,$ then 
\[
\dim \{V \in G(n,m) : \pi_{V}(E) \text{ has empty interior }\} \leq  m(n - m) -(s-2m).
\]
\end{theorem}

Here $G(n,m)$, called Grassmanian manifold, denotes the collection of all the $m$-dimensional linear subspaces of $\mathbb{R}^{n}$. The notation $\dim E$ denotes the Hausdorff dimension of the set $E$, and $\pi_{V}: \RR^{n}\rightarrow V$ stands for the orthogonal projections onto $V$. Recalling that the Grassmanian  manifold $G(n,m)$ has dimension $m(n-m)$, this is why we compare with it in the above estimates. Note that the upper bound in $(b)$ is sharp, while the knowledge for the upper bounds in $(a)$ and $(c)$ are not complete, see Mattila \cite[Chapter 5]{Mattila} for more details.

Recently there has been a growing interest in studying finite field version of some classical problems arising from Euclidean spaces. For instance, there are finite field  Kakeya sets (also called Besicovitch sets) \cite{Dvir}, \cite{Green}, \cite{Wolff}, \cite{Zhang}, there are  finite field Erd\H{o}s/ Falconer distance problem \cite{IsoevichRudnev}, \cite{Tao1}, etc.  


Motivated by the above works, we study the  projections in vector spaces over finite fields.  We show some  notations first. Let $\mathbb{F}_p$ denote the finite field with $p$ elements where $p$ is prime, and $\F$ be the $n$-dimensional vector space over this field. The number of the $m$-dimensional linear subspaces of $\F$ is ${n \choose m}_{p}$ which is called Gaussian coefficient, see \cite{Cameron}, \cite{Ko} for more details. Note that to obtain a $m$-dimensional subspace, it is sufficient to choose  $m$ linear independent vectors. For the first vector we have $p^{n}-1$ choices from $\F$ (except the zero vector);  for the second vector we have $p^{n}-p$ choices to make that the second vector is independent to the first one; and so on. In the end we have $(p^{n}-1)(p^{n}-p)\cdots (p^{n}-p^{m-1})$ amount of  choices. Note that for each $m$-dimensional subspace, there are amount of $(p^{m}-1)(p^{n}-p)\cdots (p^{m}-p^{m-1})$ choices which generate (or span) the same subspace. It follows that (see  \cite[Theorem 6.3]{Cameron}, \cite{Ko} for more details)
\begin{equation}\label{eq:bino}
{n \choose m}_{p}=\frac{(p^{n}-1)(p^{n}-p)\cdots (p^{n}-p^{m-1})}{(p^{m}-1)(p^{m}-p)\cdots (p^{m}-p^{m-1})}.
\end{equation}
One interesting fact of Gaussian binomial coefficient is that  
\begin{equation*}
{n \choose m}_{p} =(1+o(p))  p^{m(n-m)}.
\end{equation*}
The notation  $o(p)$ means that $o(p)$ goes to zero as $p$ goes to infinity. However, the following estimate is enough for our purpose. Throughout the paper, we assume that  the prime number $p$ is large enough  such that 
\begin{equation}\label{eq:condition}
p^{m(n-m)} \leq {n \choose m}_{p} \leq 2  p^{m(n-m)}.
\end{equation}
Note that the exponent $m(n-m)$ is the dimension of the real Grassmanian manifold. For convenience, we  use the  same notation $G(n,m)$ to denote all  the $m$-dimensional linear subspaces of $\F$.

Before we show the definition of  projections in vector spaces over finite fields, let's recall the orthogonal projections in Euclidean spaces. Let $E\subset \RR^{n}$ and $V$ be a subspace of $\RR^{n}$. Then the orthogonal projection of $E$ onto $V$ is defined as 
\[
\pi_{V}(E)=\{x\in V: (x+V^{\perp})\cap E \neq \emptyset\}
\]  
where $V^{\perp}$ is the orthogonal complement of $V$. Note that the vector space $\F$ is not an inner product space (in general). For instance the Lagrange's (four square) theorem, every nature number is the sum of four squares, implies that $\mathbb{F}_{p}^{n}$ is never an inner product space for $n\geq 4$. Therefore, the word  orthogonal make no sense  in these spaces. Thus we need a new way to define `projection' in $\F$. The following is one of these choices. Let $E$ be a subset of  $\F$ and $W$ be a non-trivial subspace of $\F$. Let  $\pi^{W}(E)$ denote the collection of cosets of $W$ which intersect $E$, i.e.,
\[
\pi^{W}(E)=\{x+W: E\cap (x+W) \neq \emptyset, x\in \F\}.
\]
In this paper we are interested in the cardinality of $\pi^{W}(E)$. Let $|J|$ denote the cardinality of a set $J$. Observe that  if $E\subset \RR^{n}$ is a finite set (i.e., $|E|<\infty$) then
\[
|\pi_{W^{\perp}}(E)|=|\pi^{W}(E)|.
\]
Observe that for any set $E \subset \F$ and $W\in G(n,n-m)$ (Lagrange's group theorem),
\[
|\pi^{W}(E)|\leq \min\{|E|, p^{m}\}.
\]
In analogy of Theorem \ref{thm:M}, we have  the following finite fields version. 

%



\begin{theorem}\label{thm:maintheorem}
Let $E \subset \F$.  Then 

(a) for any $N<\frac{1}{2}|E|$, 
\[
|\{W\in G(n,n-m): |\pi^{W} (E)|\leq N\}| \leq  4 p^{(n-m)m-m}N;
\]

(b) for any $\delta \in (0,1)$, 
\[
|\{W\in G(n,n-m): |\pi^{W} (E)|\leq \delta p^{m}\}| \leq  2\left(\frac{\delta}{1-\delta}\right) p^{m(n-m)+m}|E|^{-1}.
\]
\end{theorem}

We note that Theorem \ref{thm:maintheorem} $(a)$ for $m=n-1$ follows  from Orponen's pair argument \cite[Estimate (2.1)]{Orponen0}), and  Theorem \ref{thm:maintheorem} $(b)$ for $m=1,n=2$  follows from  Murphy and Petridis \cite[Corollary 1]{MurphyPetridis}. We immediately  have the following corollary via the special choices of $N$ and $\delta$ in Theorem \ref{thm:maintheorem}. 

\begin{corollary}\label{cor:mainclaim}
Let $E \subset \F$ with $|E|=p^{s}$. 

(a) If $s\leq m$ and $t \in (0, s]$, then
\[
| \{W \in  G(n,n-m) : |\pi^{W} (E)| \leq p^{t}/10 \} \leq \frac{1}{2} p^{m(n - m) -(m - t)}.
\]

(b) If $s> m$, then
\[
| \{W \in  G(n,n-m) : |\pi^{W} (E)| \leq   p^{m}/ 10 \}| \leq \frac{1}{2} p^{m(n - m) -(s-m)}.
\]

(c) If $s>2m$, then 
\[
| \{W \in  G(n,n-m) :|\pi^{W}  (E)|\neq p^{m} \}| \leq 4 p^{m(n - m) -(s-2m)}.
\] 
\end{corollary}

The Corollary \ref{cor:mainclaim} $(c)$ follows by the choice of $\delta=\frac{p^{m}-1}{p^{m}}$, and an easy fact that  if $|\pi^{W}(E)|>p^{m}-1$  then $|\pi^{W}(E)|=p^{m}$ (i.e. $E$ intersects each coset of $W$).  Note that the exponents in Corollary \ref{cor:mainclaim} are the same as in Theorem \ref{thm:M}, and  $|\pi^{W}(E)|=p^{m}$ correspondence to the existence of interior in Theorem \ref{thm:M} $(c)$. 


I do not know whether these bounds are sharp. We formulate  the following finite field version of one conjecture for the dimension bound of the exceptional set in the Euclidean plane. For more backgrounds, and partial improvements on  the size of the exceptional set in Euclidean plane, see \cite[Chapter 5]{Mattila},  \cite[Theorem 1.2]{Oberlin}, \cite{Orponen0} \cite[Proposition 1.11]{Orponen1}, \cite{Orponen2}, \cite{Orponen3}.   
\begin{conjecture}
For any $s/2\leq t\leq s\leq 1$ and $E\subset \mathbb{F}_{p}^{2}$ with  $p^{s}/C\leq |E|\leq C p^{s}$. Then 
\[
| \{W \in  G(2,1) : |\pi^{W} (E)|\leq p^{t}  \} |\leq C(s,t, C) p^{2t-s}.
\]
\end{conjecture}

 In Euclidean space, there are various random fractal sets which do not have exceptional set in the projection theorem, see \cite{SS} for more details and reference therein. For the finite field case, we study the projections of  random sets in $\F$ (percolation on $\F$). We have the following results.

\begin{theorem}\label{thm:small}
For any $0<s\leq m$, there is a positive number $p_{0}=p_{0}(n,m,s)$ such that for any prime number $p\geq p_0$, there exists a subset $E\subset \F$ with $p^{s}/2\leq |E|\leq 2p^{s}$ such that,
\[
|\pi^{W}(E)|\geq |E|/24 \text{ for all $W\in G(n,n-m)$ }.
\]
\end{theorem}

\begin{theorem}\label{thm:large}
For any $m<s\leq n$, there is a positive number $p_{0}=p_{0}(n,m,s)$ such that for any prime number $p\geq p_0$, there exists a subset $E \subset \F$ with $p^{s}/2\leq |E|\leq 2p^{s}$ such that, 
\[
|\pi^{W}(E)| = p^{m} \text{ for all $W\in G(n,n-m)$ }.
\]
\end{theorem}

Fourier analysis plays an important role in many topics in fractal geometry. Furthermore there are some results which all known proofs depend on Fourier transforms, for example the statement $(c)$ of Theorem \ref{thm:M}. Note that 
the proof for  Theorem \ref{thm:maintheorem} $(b)$ also depends on the (discrete) Fourier transformation.

\begin{proposition}\label{pro:Fourier}
Let $E\subset \F$ with $|\E(\xi)|\leq C |E|^{\alpha}$ for all $\xi\neq 0$ where $\alpha \in [1/2, 1)$ and $C$ is a positive constant.

(a) If $|E|\leq C_1 p^{\frac{m}{2-2\alpha}}$, then 
\[
|\pi^{W}(E)| \geq C_{2}|E|^{2-2\alpha} \text{ for all $W\in G(n,n-m)$ }.
\]

(b) If $|E|> C_1 p^{\frac{m}{2-2\alpha}}$, then  
\[
|\pi^{W}(E)| \geq p^{m}/2 \text{ for all $W\in G(n,n-m)$}.
\]

(c) If $|E|> C_3 p^{\frac{m}{1-\alpha}}$, then 
\[
|\pi^{W}(E)|=p^{m} \text{ for all } W\in G(n,n-m).
\]

The constants $C_{1}, C_{2}, C_{3}$ only depend on the constant $C$ and $\alpha$.  
\end{proposition} 

Here and in what follows, $\xi\neq 0$ means that $\xi$ is a non-zero vector in $\F$. For $E\subset \F$, we  simply write 
$E(x)$  for the
characteristic function of $E$, and $\widehat{E}$ it's discrete Fourier transform.

Now we  discuss the condition  that  $\alpha \in [1/2,1)$. By the definition of Fourier transformation (see \eqref{eq:dede}), for any subset $E\subset \F$, 
\begin{equation*}
 |\E(\xi)|\leq |E| \text{ for all } \xi \neq 0.
\end{equation*} 
For the case $\alpha=1/2$, Iosevich and Rudnev \cite{IsoevichRudnev} called these sets Salem sets. To be formal, a subset $E\subset \F$ is called a  $(C, s)$ Salem set  if $p^{s}/C\leq |E|\leq Cp^{s}$ and for any $\xi\neq 0$,
\[
|\E(\xi)|\leq C \sqrt{|E|}.
\]
Iosevich and Rudnev \cite{IsoevichRudnev} introduced the finite fields Salem sets for their study of  Erd\H{o}s/Falconer distance problem in  vector spaces over finite fields. Note that this is a finite fields version of Salem sets in Euclidean spaces, see \cite[Chapter 3]{Mattila} for more details on Salem sets in Euclidean spaces. Roughly speaking the Fourier coefficients  of Salem sets   
have  the `best possible' upper bound. This follows by the Plancherel identity, 
\[
\sum_{\xi \in \F}|\E(\xi)|^{2}=p^{n}|E|.
\]
To be precise, let $E\subset \F$ with $|E|\leq p^{n}/2$ and $|\E(\xi)|\leq  |E|^{\alpha}$ for any $\xi\neq 0$. Then 
\[
p^{n}|E| -|E|^{2}  = \sum_{\xi \neq 0}|\E(\xi)|^{2}\leq  p^{n}|E|^{2\alpha}.
\]
It follows that 
\[
1/2\leq |E|^{2\alpha-1}.
\]
Thus $\alpha\geq 1/2$ provided $E$ is large enough. See \cite[Proposition 2.6]{Babai}, \cite{IsoevichRudnev} for more details.

Observe that Proposition \ref{pro:Fourier}  implies that finite field Salem sets do not have  exceptional directions in the Corollary \ref{cor:mainclaim}. Furthermore,  if there exists $(C, s)$ Salem set  where $C$ does not depend on $p$, then by Proposition \ref{pro:Fourier} we can obtain  Theorems \ref{thm:small}-\ref{thm:large}. 
However, it seems that  the only known examples of Salem sets in $\F$ are the discrete paraboloid and the discrete sphere, and both the size of the  discrete paraboloid and the discrete sphere in $\F$ are roughly $p^{n-1}$, see \cite{IsoevichRudnev} for more details. It is natural to ask that does there exists $(C, s)$ Salem set in $\F$ for any (large) prime number $p$. The above results and  \cite[Problem 20]{Mattila2004} suggest that there exists $(C, s)$ Salem set only if $s$ is an integer. If we loose the condition in the definition of Salem sets, the author \cite{Chen2017}, Hayes \cite{Hayes} showed the existence of  (weak) Salem sets in any given size via the random subsets of $\F$. A set $E$ is called  (weak) Salem set if there is a constant $C$ (does not depend on $p$) such that 
\[
|\E(\xi)|\leq C \sqrt{|E| \log p}  \text{  for all } \xi \neq 0.
\] 
See Chen \cite{ChenX} for a random construction of Salem set in Euclidean space (with a $\log$ factor also).

 
The structure of the paper is as follows. In Section \ref{sec:preliminaries}, we set up  notation and  lemmas for later use. We prove Theorem \ref{thm:maintheorem}, Theorems \ref{thm:small}-\ref{thm:large}, and Proposition \ref{pro:Fourier} in Section \ref{sec:main}, Section\ref{sec:random}, and Section \ref{sec:Fourier} respectively. In the last section, we extend an identity of 
Murphy and Petridis \cite{MurphyPetridis}. We give another definition of projections in $\F$, and show that our results still holds under this definition.

\section{preliminaries} \label{sec:preliminaries}

For $1\leq m\leq n$, recall that $G(n,m)$ stands for all the $m$-dimensional linear subspaces of $\F$. Let $A(n,m)$ denote the family of all $m$-dimensional planes, i.e., the translation of some $m$-dimensional subspace. Let $W\in G(n,n-m)$, then observe by Lagrange's group theorem that there are $p^{m}$ cosets of $W$. Let $x_{W,j}+W, 1\leq j\leq p^{m}$
be the different cosets of $W$. Let $G \subset G(n,n-m)$, and define 
\[
G'= \{x_{W,j}+  W: 1\leq j\leq p^{m}, W \in G\}. 
\]

\noindent{\bf Outline of the method.} The method is an adaptation of the  counting pairs argument of Orponen \cite[Estimate (2.1)]{Orponen0} to our setting.  Let  $E\subset \FF_{p}^{n}$ and $W\in G(n,n-m)$, then $|E| =\sum_{j=1}^{p^{m}} |E\cap (x_{W,j}+W)|,$
and the Cauchy-Schwarz inequality  implies  
\begin{equation}\label{eq:pairs}
|E|^{2}\leq |\pi^{W}(E)|\sum_{j=1}^{p^{m}} |E\cap (x_{W,j}+W)|^{2}.
\end{equation}
Note that $|E\cap (x_{W,j}+W)|^{2}$ is the amount of pairs of $E$ inside $x_{W,j}+W$. Let $N\leq p^{m}$, define 
\[
\Theta=\{W\in G(n,n-m): |\pi^{W} (E)|\leq N\}.
\]
Summing two sides over 
$W\in \Theta$ in estimate $\eqref{eq:pairs}$, we obtain 
\begin{equation}\label{eq:argument}
|\Theta| |E|^{2} \leq N\sum_{W\in \Theta}\sum_{j=1}^{p^{m}}|E\cap (x_{W,j}+W)|^{2}:=N\e(E,\Theta').
\end{equation}
Therefore, the left problem is to estimate $\e(E, \Theta')$. 

\subsection{Counting pairs and energy arguments}
Motivated by the above pairs argument of Orponen \cite[Estimate (2.1)]{Orponen0}, an incidence identity of Murphy and Petridis \cite[Lemma 1]{MurphyPetridis}, and the additive energy in additive combinatorics \cite[Chapter 2]{TaoVu},  we give the following definition which plays a key role for the proof of Theorem \ref{thm:maintheorem}. 

\begin{definition}
Let $E\subset \F$ and $ \mathcal{A} \subset A(n,k)$. Define the (generalized) energy of $E$ on $\mathcal{A}$ as
\[
\e(E, \mathcal{A}) =\sum_{W\in \mathcal{A}} |E\cap W|^{2}.
\] 
\end{definition}

\begin{remark} Recalling the additive energy, if $A, B\subset \FF_{p}$ then the  additive energy $E(A, B)$ between $A$ and $B$ is defined as  the quantity
\begin{equation*}
E(A, B) = |\{(a, a',
 b, b') \in A \times A \times B \times B : a + b = a' + b'\}|.
\end{equation*}
Observe that by our notation 
\begin{equation*}
E(A, B)=\e(A\times B, \mathcal{L}),
\end{equation*}
where 
\[
\mathcal{L}=\{\ell_{k}\}_{k=0}^{p-1}, \text{ and } \ell_{k}=\{(x,y)\in \mathbb{F}_{p}^{2} :x+y=k\}.
\]
For more details on additive energy, see  
\cite{TaoVu}.
\end{remark}

%

\subsection{Discrete Fourier transformation}
In the following we collect the basic facts about Fourier transformation which related to our setting. For more details on discrete Fourier analysis, see Green \cite{Green}, Stein and Shakarchi \cite{Stein}. Let $f : \F\longrightarrow \mathbb{C}$ be a complex value function. Then the Fourier transform of $f$ at $\xi \in \F$ is defined as 
\begin{equation}\label{eq:dede}
\widehat{f}(\xi)=\sum_{x\in \F} f(x)e(-x\cdot \xi),
\end{equation}  
where $e(-x \cdot \xi)=\ee$ and the dot product 
\[
 x\cdot\xi =x_1\xi_1+\cdots +x_n\xi_n\,(\text{mod}\, p).
 \] 
Recall the following Plancherel identity, 
\begin{equation*}
 \sum_{\xi \in \F}|\widehat{f}(\xi)|^{2}=p^{n}\sum_{x\in \F} |f(x)|^{2}.
\end{equation*} 
Specially for the subset $E\subset \F$, we have 
\[
\sum_{\xi \in \F} \m\E\m^{2}=p^{n}\m E\m.
\]

In the following of this subsection we intend to establish the following  (`Plancherel identity on subspaces') Lemma \ref{lem:fff}. To be formal we give some notation first. For $W\in G(n,n-m)$, we define the `orthogonal complement' of $W$ as  
\[
Per(W):=\{x\in \F: x\cdot w=0 \,(\text{mod}\, p), w\in W\}.
\]
Note that unlike in the Euclidean spaces, here $W\cap Per(W)$ can be some non-trivial subspace. For  example  let $W=\text{span}\{(1,1)\}\subset \mathbb{F}_{2}^{2}$ then $Per(W)= W$. However,  the rank-nullity theorem of linear algebra (or  the solution of system of linear equations) implies  that for any subspace $W \subset\F$,
\begin{equation}\label{eq:rank}
\dim W+\dim Per(W)=n.
\end{equation}

The following result shows the connection between  $|E\cap (x_{W,j}+W)|, 1\leq j\leq p^{m}$ and the Fourier transform of $E$, and the identity \eqref{eq:kk} plays an important role  in the proof of Theorem \ref{thm:maintheorem} (b) and the proof of Proposition \ref{pro:Fourier}.
\begin{lemma}\label{lem:fff}
Use the above notation.  We have 
\begin{equation}\label{eq:kk}
\sum_{j=1}^{p^{m}} | E \cap (x_{W,j}+W)|^{2}=p^{-m}\sum_{\xi\in Per(W)} |\E(\xi)|^{2}.
\end{equation}
\end{lemma}
\begin{proof}
Let $\xi \in Per(W)$ then
\begin{equation*}
\begin{aligned}
\widehat{E}(\xi)&=\sum_{x\in \F}E(x)e(-x\cdot \xi)\\
&=\sum_{j=1}^{p^{m}} \sum_{ w\in W}E(x_{W,j}+w)e(-(x_{W,j}+w)\cdot\xi)\\
&=\sum_{j=1}^{p^{m}} |E\cap (x_{W,j}+W)|e(-x_{W,j}\cdot\xi).
\end{aligned}
\end{equation*}
It follows that 
\begin{equation*}
\begin{aligned}
|\widehat{E}(\xi)|^{2}&=\sum_{j=1}^{p^{m}}|E\cap (x_{W,j}+W)|^{2} \\
&+\sum_{j\neq k} |E\cap (x_{W,j}+W)||E\cap (x_{W,k}+W)|e(-(x_{W,j}-x_{W,k})\cdot\xi).
\end{aligned}
\end{equation*}
Note that $(x_{W,j}-x_{W,k}) \notin W $ for any $j\neq k$. Together with the following Lemma \ref{lem:character}, we obtain
\[
\sum_{\xi\in Per(W)} e(-(x_{W,j}-x_{W,k})\cdot\xi)=0.
\] 
Thus we complete the proof.
\end{proof}

\begin{lemma}\label{lem:character}
Let $V\in G(n,k)$ and $x\not\in V$. Then 
\[
\sum_{y\in Per(V)} e(-x \cdot y)=0.
\]
\end{lemma}
\begin{proof}
We claim that there exists $y_{0}\in Per(V)$ such that $y_{0}\cdot x\neq 0$. Suppose that $x\cdot y=0$ for any $y\in Per(V)$. It follows that 
\[
Per(V)\subset Per(V\cup \{x\}).
\]
Combing with the estimate \eqref{eq:rank} (rank-nullity theorem), we obtain  that $n-k\leq n-k-1$ which is impossible.   Since $Per(V)$ is a subspace and $e(-y_{0}\cdot x)\neq 1$, we obtain 
\[
e(-y_{0}\cdot x) \sum_{y\in Per(V)} e(-x \cdot y)=\sum_{y\in Per(V)} e(-x \cdot y),
\]
and hence $\sum_{y\in Per(V)} e(-x \cdot y)=0.$ 
\end{proof}

\begin{remark}
We do not know if the Lemma \ref{lem:character} also holds for  vector spaces over general finite fields. For that case we  will take nonprincipal character instead of $e^{\frac{2 \pi i x}{p}}$. We note that the Lemma \ref{lem:character} is the only place in this paper where the prime field $\mathbb{F}_{p}$ is needed.  
\end{remark}

%

\subsection{Counting subspaces of $\F$}

In the following, we collect some basic identities for ${n \choose m}_{p}$  for our   later use.  For more details see \cite[Chapter 6]{Cameron}, \cite{Ko}.

\begin{lemma} \label{lem:iidentity}
Let  $1\leq m\leq n$.
 
(1) ${n \choose 0}_{p}:={n \choose n }_{p}=1$, \, ${n \choose m}_{p}={n \choose n-m }_{p}.$

(2) ${n \choose m}_{p}={n-1 \choose m }_{p} +p^{n-m}{n-1 \choose m-1 }.$

(3) ${n \choose m}_{p}={n-1 \choose m-1 }_{p} +p^{m}{n-1 \choose m }.$
\end{lemma}
%

\begin{lemma}\label{lem:subspace}
Let $\xi$ be a non-zero vector of $\F$. 

(1) $|\{V\in G(n,m): \xi\in V\}|={n-1 \choose m-1}_{p}$. 

(2) $|\{V\in G(n,m): \xi\in Per(V)\}|={n-1 \choose m}_{p}$.
\end{lemma}
\begin{proof}
First note that if $m=1$ then $(1)$ holds. For the case $2\leq m<n$, note that to obtain a $m$-dimensional subspace which contains the given vector $\xi$, it is sufficient to choose another $m-1$  vectors such that these $m-1$ vectors and the vector $\xi$ span a $m$-dimensional subspace. For the choice of  the first vector, we have $p^{n}-p$ choices from $\F$ (except the vectors from the subspace of $\F$ which spanned by $\xi$).  For the choice of the second vector, we have $p^{n}-p^{2}$ choices to make sure that the second vector, the first vector, and the  vector $\xi$ span a $3$-dimensional subspace. We continue to do this until we choose $m-1$ vectors. In the end we have $(p^{n}-p)(p^{n}-p^{2})\cdots (p^{n}-p^{m-1})$ amount of  choices. Note that for each $m$-dimensional subspace which contains vector $\xi$, there are amount of $(p^{m}-p)(p^{n}-p^{2})\cdots (p^{m}-p^{m-1})$ choices which generate (or span) the same subspace. It follows that (see \eqref{eq:bino} for the definition of ${n \choose m}_{p}$) the  amount of $m$-dimensional plane which contain $\xi$ is 
\[
\frac{(p^{n}-p)(p^{n}-p^{2})\cdots (p^{n}-p^{m-1})}{(p^{m}-p)(p^{m}-p^{2})\cdots (p^{m}-p^{m-1})}={n-1 \choose m-1}_{p}.
\]

To establish $(2)$, first note that $\dim Per(\xi)=n-1$. Observe that 
\[
\{V\in G(n,m): \xi\in Per(V)\}
\] 
is the collection of all the $m$-dimensional subspace of $Per(\xi)$, which the conclusion follows.
\end{proof}


\section{Proof of Theorem \ref{thm:maintheorem} }\label{sec:main}

Let $\Theta \subset G(n,n-m)$. Recall that  
\[
\Theta'= \{x_{W,j}+  W: 1\leq j\leq p^{m}, W \in \Theta\}. 
\]

\begin{lemma}\label{lem:keylemma}
Let $E\subset \F$, $\Theta \subset G(n, n-m)$. Then
\begin{equation}
\e(E, \Theta') 
\leq \min \left\{ |E||\Theta|+2|E|^{2}  p^{(n-m-1)m}, 2| E| p^{(n-m)m}+| E| ^{2}| \Theta| p^{-m} \right\}.
\end{equation}
\end{lemma}
\begin{proof}
We first show the estimate
\[
\e(E, \Theta')\leq |E||\Theta|+2|E|^{2}  p^{(n-m-1)m}.
\] 
Let $x\in \F$. Since for any $W\in \Theta$ there is only one coset of $W$ which contains $x$, we obtain   
\[ 
\sum_{V\in \Theta'} V(x)=|\Theta|.
\]
By Lemma \ref{lem:subspace}, there are   ${n-1 \choose n- m-1}_{p}$ amount of  $(n-m)$-dimensional subspaces containing the given non-zero vector of $\F$. Then
\begin{equation*}
\begin{aligned}
\e(E, \Theta')
&=\sum_{V\in \Theta'} \left(\sum_{x\in E}V(x) \right)^{2}\\
&=\sum_{V\in \Theta'} \left(\sum_{x\in E}V(x)+\sum_{x\neq y \in E} V(x)V(y) \right)\\
&\leq |E||\Theta|+|E|(|E|-1){n-1 \choose n-m-1}_{p}\\
&\leq |E||\Theta|+2|E|^{2}p^{(n-m-1)m}.
\end{aligned}
\end{equation*}

Now we turn to the other estimate. Applying Lemma \ref{lem:subspace} we have
\begin{equation*}
\begin{aligned}
\sum_{W\in \Theta} \sum_{\xi\in Per(W)}|\E(\xi)|^{2}-|\Theta||E|^{2} &\leq { n-1 \choose m-1}_{p}\sum_{\xi \in \F\backslash \{0\}} |\E(\xi)|^{2}\\
&\leq { n-1 \choose m-1}_{p} p^{n}|E|.
\end{aligned}
\end{equation*}
Together with  Lemma \ref{lem:fff}, we obtain 
\begin{equation*}
\begin{aligned}
\e(E, \Theta')
&=\sum_{W\in \Theta} \sum_{j=1}^{p^{m}}|E\cap(x_{W,j}+W)|^{2}\\
& =p^{-m}\sum_{W\in \Theta} \sum_{\xi\in Per(W)}|\E(\xi)|^{2}\\
&=p^{-m}\left(\sum_{V\in \Theta} \sum_{\xi\in V}|\E(\xi)|^{2}-|\Theta||E|^{2}+|\Theta||E|^{2}\right)\\
&\leq p^{-m}p^{n}|E| {n-1 \choose m-1}_{p} +|E|^{2}|\Theta| p^{-m}\\
&\leq 2p^{(n-m)m}|E|+|E|^{2}|\Theta| p^{-m}.
\end{aligned}
\end{equation*}
Thus we complete the proof.
\end{proof}

\begin{remark}
Let  $|E|\leq p^{m}$ and $\Theta\subset G(n,n-m)$. Then by the estimate  \eqref{eq:condition} and Lemma \ref{lem:iidentity} (1)     
\[
 |\Theta| \leq { n\choose n-m}_{p}\leq 2p^{m(n-m)}.
 \]
Thus
\[
|E||\Theta|+2|E|^{2}  p^{(n-m-1)m} \leq 
2| E| p^{(n-m)m}+| E| ^{2}| \Theta| p^{-m}.
\]
Therefore, in the proof of Theorem \ref{thm:maintheorem}, if $|E|\leq p^{m}$ then we  use the  estimate 
\begin{equation}\label{eq:a}
\e(E, \Theta') \leq |E||\Theta|+2|E|^{2}  p^{(n-m-1)m}.
\end{equation} 
For the case $|E|>p^{m}$, we   use the estimate 
\begin{equation}\label{eq:b}
\e(E, \Theta') \leq 2| E| p^{(n-m)m}+| E| ^{2}| \Theta| p^{-m}.
\end{equation}

\end{remark}

\begin{proof}[Proof of Theorem \ref{thm:maintheorem}]

First we prove  $(a).$ Let 
 \[
\Theta=\{W\in G(n,n-m): |\pi ^{W} (E)|\leq N\}.
\]
Applying the estimates \eqref{eq:argument} (outline of the method), \eqref{eq:a}, we obtain
\begin{equation*}
\begin{aligned}
|\Theta||E|^{2} &\leq \e(E,\Theta')N\\
&\leq (|E||\Theta|+2|E|^{2}p^{(n-m-1)m})N.
\end{aligned}
\end{equation*}
It follows that  (recall  that $N\leq |E|/2$)
\[
|\Theta|\leq 4p^{(n-m)m-m}N.
\]

Now we prove  $(b)$. For $\delta \in (0,1)$, let 
\[
\Theta=\{W\in G(n,n-m): |\pi^{W} (E)|\leq \delta p^{m}\}.
\]
Applying the estimates \eqref{eq:argument}, \eqref{eq:b}, we obtain 
\begin{equation*}
\begin{aligned}
| \Theta||E|^{2} &\leq \e(E,\Theta')\delta p^{m}\\
&\leq (2|E|p^{(n-m)m}+|E|^{2}|\Theta| p^{-m})\delta p^{m},
\end{aligned}
\end{equation*}
and
\[
|\Theta||E|(1-\delta)\leq 2\delta p^{(n-m)m+m}.
\]
Then
\[
|\Theta| \leq 2 \left(\frac{\delta}{1-\delta}\right)p^{(n-m)m+m}|E|^{-1}.
\]
Thus we complete the proof.
\end{proof}

\section{Proofs of Theorems \ref{thm:small} and \ref{thm:large}}\label{sec:random}

\subsection{Percolation on $\FF_{p}^{d}$ }

The random model we used here is related to many other well known models, for example: 
Erd\H{o}s-R\'enyi-Gilbert model in random graphs,  percolation theory on the graphs, and  Mandelbrot percolation in fractal geometry. We show this model on $\F$ in the following. For an application of this model to find sets with small Fourier coefficient, see \cite[Theorem 5.2]{Babai}, \cite{Chen2017}. 

Let $0<\delta<1$. We choose each point of $\F$ with probability $\delta$ and remove it with probability $1-\delta$, all choices being independent of each other. Let $E=E^{\omega}$ be the collection of these chosen points. Let $\Omega=\Omega (\F, \delta)$ be our probability space which consists of all the possible sets $E^{\omega}$.

We prove Theorems  \ref{thm:small} and \ref{thm:large} by choose $\delta=p^{s-n}$ in the above  model, and show that the random set $E$  has the desired properties with high probability when $p$ is large enough. For convenience, we formulate a special large deviations estimate in the following. For more background and details on large deviations estimate, see Alon and Spencer \cite[Appendix A]{Alon}.

\begin{lemma}[Chernoff bound]\label{lem:largedeviations}
Let $\{X_j\}_{j=1}^N$ be a sequence independent Bernoulli random variables with success probability $\delta'$. Let $\mu=N\delta'$, then    
\[
\PP\left(\sum^N_{j=1} X_j < \mu  /2 \right)\leq e^{-\mu /16}.
\]
\end{lemma}

%

\subsection{Projections of random sets}

\begin{proof}[Proof of Theorem \ref{thm:small}]
Let $\delta=p^{s-n}$. We consider the random model $\Omega(\F, \delta).$ 
Let $W\in G(n,n-m)$. Observe that $|\pi^{W}(E)|$ is a binomial distribution with parameters $p^{m}$ and $\delta'$ where 
\[
\delta'=1-(1-\delta)^{p^{n-m}}.
\]
Let $\mu=p^{m}\delta'$. Since $1+x\leq e^{x}$ holds for all $x$, $e^{x} \leq 1+x +5x^{2}/6$ holds for $|x|\leq 1$, and $s\leq m$, we obtain 
\begin{equation}\label{eq:mu}
\begin{aligned}
\mu= p^{m}\delta'&\geq p^{m}(1-e^{-\delta p^{n-m}})
\\
&=p^{m}(1-e^{-p^{s-m}})\geq p^{m}(p^{s-m}-5p^{2(s-m)}/6)\\
&\geq p^{s}/6.
\end{aligned}
\end{equation}

Applying Lemma \ref{lem:largedeviations} to $|\pi^{W}(E)|$ we have 
\[
\PP\left(|\pi^{W}(E)| < \mu  /2 \right)\leq e^{-\mu /16}.
\]
Since there are ${ n \choose m}_{p}$  elements  of  $G(n, n-m)$, we have 
\begin{equation}\label{eq:p}
\begin{aligned}
\PP (  \text{exists }  W \in G(n,n-m) &\text{ such that } |\pi^{W}(E)|\leq \mu/2   )\\ 
& \leq { n \choose m}_{p} e^{-\mu/16}\\
&\leq 2p^{m(n-m)}e^{-p^{s} /96}\\&\rightarrow 0 \text{ as } p\rightarrow \infty.
\end{aligned}
\end{equation}

Note that $p^{s}/2 \leq |E| \leq 2 p^{s}$ with high probability ($>1/2$). This follows by applying Chebyshev's inequality, which says that
 
\begin{equation*}
\begin{aligned}
\PP(||E| - p^{n}\delta|> \frac{1}{2}p^{n}\delta)&\leq \frac{4p^{n}\delta(1-\delta)}{(p^{n}\delta)^{2}}\\
&\leq  \frac{4}{p^{s}}\rightarrow 0 \text{ as } p \rightarrow \infty.
\end{aligned}
\end{equation*}
Together with the estimates \eqref{eq:p} and \eqref{eq:mu}, we conclude that there exists $E\in \Omega(\F, \delta)$ with
\[
p^{s}/2\leq |E|\leq 2p^{s} \text{ and }|\pi^{W}(E)|> \mu /2\geq p^{s}/12 \text{ for all $W\in G(n,n-m)$ } 
\]
when $p$ is large enough. Thus we complete the proof.
\end{proof}

\begin{remark}
Let  $s\leq m$ and $\delta=p^{s-n}$. Then the above proof implies that for ``almost every" $E\in \Omega(\F, \delta)$ we have 
\[
p^{s}/2\leq |E|\leq 2p^{s} \text{ and }|\pi^{W}(E)|\geq |E|/24 \text{ for all $W\in G(n,n-m)$}. 
\]
Roughly speaking, there is no exceptional projections for almost every $E\in \Omega(\F, \delta)$. 
\end{remark}

\begin{proof}[Proof of Theorem \ref{thm:large}]
Let $\delta=p^{s-n}$. We again consider the random model $\Omega(\F, \delta)$. Let  $V\in A(n,n-m)$. Note that $|E\cap V|$ is a binomial distribution with parameters $p^{n-m}$ and $\delta$. Thus 
\begin{equation*}
\begin{aligned}
\PP(E \cap V =\emptyset)&=(1-\delta)^{p^{n-m}}\\
&\leq e^{-\delta p^{n-m}}=e^{-p^{s-m}}.
\end{aligned}
\end{equation*}
Observe that  
\begin{equation*}
|A(n,n-m)|=p^{m} {n \choose n-m}_{p}\leq 2p^{m(n-m+1)}.
\end{equation*} 
Now we have (recalling $s>m$)
\begin{equation*}
\begin{aligned}
\PP( \text{exists } V &\in A(n,n-m) \text{ such that } E \cap V = \emptyset)\\
& \leq 2p^{m(n-m+1)}e^{-p^{s-m}}\rightarrow 0 \text{ as } p\rightarrow \infty.
\end{aligned}
\end{equation*}
It follows that  
\begin{equation*}
\begin{aligned}
\PP(\text{exists } W\in G(n,n-m) &\text{ such  that } |\pi^{W}(E)| <p^{m}) \\ &\rightarrow 0 \text{ as } p\rightarrow \infty.
\end{aligned}
\end{equation*} 
Again by Chebyshev's inequality, we have  $p^{s}/2 \leq |E| \leq 2 p^{s}$ with high probability  ($>1/2$) provided $p$ is large enough. Thus we complete the proof.
\end{proof}

\begin{remark}
Let  $s> m$ and $\delta=p^{s-n}$. Then the above proof implies that for ``almost every" $E\in \Omega(\F, \delta)$ we have 
\[
p^{s}/2\leq |E|\leq 2p^{s} \text{ and }|\pi^{W}(E)|= p^{m} \text{ for all $W\in G(n,n-m)$}. 
\]
Roughly speaking, there is no exceptional projections for almost every $E\in \Omega(\F, \delta)$. 
\end{remark}

\section{Proof of proposition \ref{pro:Fourier}}\label{sec:Fourier}

\begin{proof}[Proof of Proposition \ref{pro:Fourier}]
First we use the same argument as in the outline of the method. Let $W\in G(n,n-m)$, and $x_{W,j}+W, 1\leq j\leq p^{m}$ be the different cosets of $W$. Then the Cauchy-Schwarz inequality implies  
\begin{equation*}
|E|^{2}\leq |\pi^{W}(E)|\sum_{j=1}^{p^{m}}|E\cap (x_{W,j}+W)|^{2}.
\end{equation*}
Applying  Lemma \ref{lem:fff} and the Fourier decay of $E$, we obtain 
\begin{equation*}
\begin{aligned}
|\sum_{j=1}^{p^{m}}|E\cap (x_{W,j}+W)|^{2}&=p^{-m}\sum_{\xi \in Per(W)}|\E(\xi)|^{2}\\
&\leq p^{-m}(|E|^{2}+p^{m}C^{2}|E|^{2\alpha}).
\end{aligned}
\end{equation*}
Then
\begin{equation*}\label{eq:key}
|E|^{2}\leq |\pi^{W}(E)|(p^{-m}|E|^{2}+C^{2}|E|^{2\alpha}) .
\end{equation*}
It follows that if $p^{-m}|E|^{2}\leq C^{2}|E|^{2\alpha}$ then 
\[
|\pi^{W}(E)| \geq |E|^{2-2\alpha}/2C^{2}. 
\]
On the other hand, if $p^{-m}|E|^{2}> C^{2}p^{m}|E|^{2\alpha}$ then 
\[
|\pi^{W}(E)|\geq p^{m}/2.
\] 
Then $(a)$ and $(b)$ hold. Now we turn to $(c)$. Note that if $|\pi^{W}(E)|>p^{m}-1$ then $|\pi ^{m}(E)|=p^{m}.$ Thus it is sufficient to show that  
\begin{equation}\label{eq:last}
\frac{|E|^{2}}{p^{-m}|E|^{2}+C^{2}|E|^{2\alpha}} >p^{m}-1.
\end{equation}
By calculation if  
\[
|E|> p^{\frac{m}{1-\alpha}}(2C^{2})^{\frac{1}{2-2\alpha}},  
\]
then the estimate \eqref{eq:last} holds. Thus we complete the proof.
\end{proof}

\section{Further results }

In the following, we extend an identity of 
Murphy and Petridis \cite{MurphyPetridis} for general $1\leq m\leq n$. They proved the special case $m=n-1$ for the following identity \eqref{eq:mur}. We prove it in two different ways.   
The first proof essentialy comes from  \cite{MurphyPetridis}. The second proof depends on the discrete Fourier transformation.

\begin{proposition}\label{pro:i}
Let $E\subset \F$. Then 
\begin{equation}\label{eq:mur}
\e(E, A(n,m))= |E|p^{m}{n-1 \choose m}_p +|E|^{2}{n-1 \choose m-1}_{p}.
\end{equation}
\end{proposition}
\begin{proof}
First proof. Recall that we denote by $F(x)$ the characteristic function of the subset $F\subset \F$. Applying the identities in Lemma \ref{lem:iidentity} (Gaussian  coefficient) and Lemma \ref{lem:subspace}, we have
\begin{equation*}
\begin{aligned}
\e(E, A(n,m))&= \sum_{W\in A(n,m) }|E \cap W|^{2}\\
&=\sum_{W \in A(n,m)} \left(\sum_{x\in E}W(x) \right)^{2}\\
&=\sum_{W \in A(n,m)} \left(\sum_{x\in E}W(x)+\sum_{x\neq y \in E} W(x)W(y) \right)\\
&=|E|{n \choose m}_{p}+|E|(|E|-1){n-1 \choose m-1}_{p}\\
&=|E|p^{m}{n-1 \choose m}_{p}+|E|^{2}{n-1 \choose m-1}_{p}.
\end{aligned}
\end{equation*}

Now we use a Fourier approach to give a different proof. Applying  Lemma \ref{lem:fff}, Lemmas \ref{lem:iidentity}-\ref{lem:subspace}, and  Plancherel identity for subset $E\subset \F$, we obtain 
\begin{equation*}
\begin{aligned}
\e(E, A(n,m))
&=\sum_{W\in G(n,m)}\sum_{j=1}^{p^{m}}|E\cap (x_{W,j}+W)|^{2}\\
& =p^{m-n}\sum_{W\in G(n,m)} \sum_{\xi\in Per(W)}|\E(\xi)|^{2}\\
&=p^{m-n} \left({n-1 \choose m}_{p}\sum_{\xi \neq 0}|\E(\xi)|^{2} +{n \choose m}_{p}|\E(0)|^{2} \right)\\
&=|E|p^{m}{n-1 \choose m}_{p}+|E|^{2}{n-1 \choose m-1}_{p}.
\end{aligned}
\end{equation*}
Thus we complete the proof.
\end{proof}

I thank the anonymous referee for suggesting the following definition of projections in $\F$. Let $E\subset \F$ and $V\in G(n,m)$. Then the projection of $E$ to $V$ is defined as 
\[
P_{V}(E):=\{x+Per(V): E\cap(x+Per(V))\neq \emptyset, x\in \F\}.
\]
We intend to show that  our results also holds under this definition.

By using our notation we obtain
\[
P_{V}(E)=\pi^{Per(V)}(E).
\]
Note that the rank-nullity theorem clams $\dim V+ \dim Per(V)=n$. Now we show that for $W\in G(n,n-m)$ and $E\subset \F$，
\[
P_{Per(W)}(E)=\pi^{W}(E).
\]
This follows from the fact that $Per(Per(W))=W$ for any subspace $W\subset \F$. First note that the definition of $Per(W)$ implies $W\subset Per(Per(W))$. By applying rank-nullity theorem, we obtain
\[
\dim Per(Per(W))=n-\dim Per(W)=\dim W,
\] 
and hence $Per(Per(W))=W$. Therefore, for $E\subset \F$ and $N<p^{m}$, we obtain
\[
|\{V\in G(n,m): |P_{V}(E)|\leq N\}|=|\{W\in G(n,n-m): |\pi^{W}(E)|\leq N\}|.
\]
Thus we conclude that our results  holds under the above definition.

\medskip
\textbf{Acknowledgements.} 
I would like to thank Tuomas Orponen for sharing his idea on the projections of finite points in the plane.  I thank 
the anonymous referee for carefully reading the manuscript and giving excellent
comments, and thus improving the quality of this article. I am grateful  for being  supported by the Vilho, Yrj\"o, and Kalle V\"ais\"al\"a foundation.







\begin{thebibliography}{}




\bibitem{Alon} N. Alon and J. Spencer. The probabilistic method. New York: WileyInterscience,
2000.



\bibitem{Babai} L. Babai. Fourier Transforms and Equations over Finite Abelian
Groups, An introduction to the method of trigonometric sums. http://people.cs.uchicago.edu/~laci/reu02/fourier.pdf


\bibitem{Cameron} P. Cameron, The art of counting, cameroncounts.files.wordpress.com/2016/04/acnotes1.pdf


\bibitem{Chen2017} C. Chen, Salem sets in vector spaces over finite fields, preprint (2017), available at arxiv.org/abs/1701.07958.

\bibitem{ChenX} X. Chen, Sets of Salem type and sharpness of the $L^{2}$- Fourier restriction theorem, to appear in Trans.
Amer. Math. Soc., http://arxiv.org/abs/1305.5584


\bibitem{Dvir} Z. Dvir, On the size of Kakeya sets in finite fields, J. Amer. Math. Soc. 22 (4) (2009) 1093-1097


\bibitem{Falconer1982} K.J. Falconer, Hausdorff dimension and the exceptional set of projections, Mathematika 29 (1982), 109-115.


\bibitem{Falconer} K. Falconer, J. Fraser and X. Jin, Sixty Years of Fractal Projections, in Fractal Geometry and
Stochastics V, Vol. 70 of Progress in Probability, 3-25


\bibitem{Green} B. Green, Restriction and Kakeya phonomena, Lecture notes (2003).


\bibitem{Hayes} T. Hayes, A Large-Deviation Inequality for Vector-valued Martingales. (see
https://www.cs.unm.edu/ hayes/papers/VectorAzuma/VectorAzuma20050726.pdf)

\bibitem{IsoevichRudnev} A. Iosevich and M. Rudnev, Erd\H{o}s distance problem in vector spaces over finite fields, Trans. Amer. Math. Soc.
359 (2007), 6127-6142.





\bibitem{Kaufman} R. Kaufman, On Hausdorff dimension of projections, Mathematika 15 (1968), 153-155


\bibitem{Ko} J. Konvalina,
A unified interpretation of the binomial coefficients, the Stirling numbers, and the Gaussian coefficients,
Amer. Math. Monthly, 107 (10) (2000), pp. 901-910

\bibitem{Marstrand}
J. M. Marstrand, Some fundamental geometrical properties of plane sets of fractional dimensions, Proc.
London Math. Soc.(3), 4, (1954), 257-302


\bibitem{Mattila1975} P. Mattila, Hausdorff dimension, orthogonal projections and intersections with planes,
Ann. Acad. Sci. Fenn. Ser. A I Math. 1 (1975), no. 2, 227-244.


\bibitem{Mattila2004} Hausdorff dimension, projections, and the Fourier transform, Publ. Mat. 48 (2004), 3-48.

\bibitem{Mattila} P. Mattila, Fourier analysis and Hausdorff dimension, Cambridge Studies in Advanced Mathematics,
vol. 150, Cambridge University Press, 2015.

\bibitem{MurphyPetridis} B. Murphy and G. petridis, A point-line incidence identity in finite fields, and applications, preprint (2016), available
at arxiv.org/abs/1601.03981.



\bibitem{Oberlin} D. Oberlin, Restricted Radon transforms and projections of planar set, Canad. Math. Bull. 55,
815-820.

\bibitem{Orponen0} T. Orponen, A discretised projection theorem in the plane, preprint (2014), available
at arxiv.org/pdf/1407.6543.pdf


\bibitem{Orponen1} T. Orponen, On the Packing Dimension and Category of Exceptional Sets of Orthogonal Projections, Ann.
Mat. Pura Appl. 194 (3) (2015), 843-880

\bibitem{Orponen2} T. Orponen, Projections of planar sets in well-separated directions, Adv. Math. 144 (8) (2016), 3419-3430

\bibitem{Orponen3} T. Orponen, Improving Kaufman’s exceptional set estimate for packing dimension, preprint (2016), available
at arXiv:1610.06745





\bibitem{Peres} Y. Peres and B. Schlag, Smoothness of projections, Bernoulli convolutions, and the
dimension of exceptions, Duke Math. J. 102 (2000), 193-251.

\bibitem{SS} P. Shmerkin and V. Suomala: Spatially independent martingales, intersections, and applications. To appear in Mem. Amer. Math. Soc.


\bibitem{Stein} E. Stein and R. Shakarchi, Fourier Analysis: An Introduction. Princeton and Oxford:
Princeton UP, 2003. Print. Princeton Lectures in Analysis.

\bibitem{Tao1}  T. Tao, Finite field analogues of Erd\H{o}s, Falconer, and Furstenberg problems.

\bibitem{TaoVu} T. Tao and V. Vu, Additive Combinatorics, Cambridge University Press.


\bibitem{Wolff} T. Wolff, Recent work connected with the Kakeya problem. Prospects in mathematics
(Princeton, NJ, 1996). pages 129-162, 1999.

\bibitem{Zhang} R. Zhang, On configurations where the Loomis-Whitney inequality is nearly sharp and applications to the Furstenberg set problem, Mathematika 61 (1) (2015), 145-161

\end{thebibliography}
\end{document}